\documentclass[11pt]{amsart}
\usepackage{amscd,amssymb,color}
\usepackage[graph,frame,poly,arc]{xy}  
\usepackage[plainpages,backref,urlcolor=blue]{hyperref}
\usepackage{xcolor}
\usepackage[dvipsnames]{xcolor}
\definecolor{darkgreen}{RGB}{0,100,0}
\theoremstyle{plain}
\newtheorem{thm}[subsection]{Theorem}
\newtheorem{lem}[subsection]{Lemma}
\newtheorem{prop}[subsection]{Proposition}
\newtheorem{cor}[subsection]{Corollary}

\theoremstyle{definition}
\newtheorem{rk}[subsection]{Remark}

\newtheorem{ex}[subsection]{Example}

\newtheorem{question}[subsection]{Question}

\numberwithin{equation}{section}
\newcommand{\M}{{\mathcal M}}

\newcommand{\C}{\mathbb{C}}
\newcommand{\PP}{\mathbb{P}}
\newcommand{\AAA}{\mathbb{A}}

\DeclareMathOperator{\Aut}{Aut}
\DeclareMathOperator{\Sing}{Sing}

\begin{document}

\title[Jacobian algebras and variation of hyperplane sections] {Jacobian algebras and variation of hyperplane sections}

	\author{Giovanna Ilardi}
		\address{Dipartimento Matematica e Applicazioni ``R. Caccioppoli'', Università degli Studi	Di Napoli “Federico II” Via Cintia - Complesso Universitario Di Monte S. Angelo 80126 - Napoli - Italia}
    \email{giovanna.ilardi@unina.it}
		\author{Abbas Nasrollah Nejad}
		\address{Department of Mathematics, Institute for Advanced Studies in Basic Sciences (IASBS), Zanjan 45137-66731, Iran}
		\email{abbasnn@iasbs.ac.ir}
		\author{Saeed Tafazolian}
		\address{Universidade Estadual de Campinas (UNICAMP) \\ Instituto de Matem\'atica, Estat\'{\i}stica e Computação Cient\'{\i}fica (IMECC) \\ Departamento de Matem\'atica \\
			Rua S\'ergio Buarque de Holanda, 651\\ 13083-970 Campinas-SP, Brazil}
		\email{saeed@unicamp.br}
\thanks{The third author was partially supported by CNPq grant no.~302774/2025-4, FAEPEX grant no.~3485/25, and FAPESP grant no.~2024/00923-6.}

\subjclass[2020]{Primary 14D20, 14J70; Secondary 13A02, 13D02, 14B05.}

\keywords{Hyperplane section, moduli space, Jacobian algebra, Lefschetz properties, isolated hypersurface singularities, projective automorphisms}

\begin{abstract}
We study the variation in moduli of hyperplane sections of a hypersurface $V(f)\subseteq\PP^n$ with at most isolated singularities. Using the Milnor
algebra $M(f)$, we give an infinitesimal quotient criterion for the hyperplane-section map
\[
\Phi(f):\PP^{n*}\dashrightarrow \M(d,n-1)
\]
to be generically finite onto its image. The passage from the infinitesimal quotient to the coarse moduli space is justified by a local GIT slice argument.
Our approach gives a Jacobian-algebraic extension of the Beauville--Patel--Riedl--Tseng theory from smooth hypersurfaces to hypersurfaces with isolated singularities. In the smooth case it recovers the Lefschetz criterion and, using recent weak Lefschetz results, gives generic finiteness for $n\geq3$ in the range $d\geq n+2$. In the singular case a new obstruction appears: a linear Jacobian syzygy, equivalently, for non-cones, a positive-dimensional projective automorphism group. After this obstruction is excluded, maximal infinitesimal variation is governed by the injectivity of the critical Lefschetz map $\ell:M(f)_{d-1}\to M(f)_d$. We apply the criterion to plane curves, surfaces in $\PP^3$, and hypersurfaces admitting singular hyperplane sections, obtaining new criteria involving nodal sections and an
application to the Schoen quintic threefold.

\end{abstract}

\maketitle
\section{Introduction}
\label{sec:introduction}
Let $S=\C[x_0,\ldots,x_n]$ and let $f\in S_d$ be a homogeneous polynomial of degree $d\geq 3$ defining the hypersurface $X=V(f)\subseteq \PP^n$. Assume that $X$ has at most isolated singularities. Then a general hyperplane $H\subseteq \PP^n$ avoids $\operatorname{Sing}(X)$ and, by Bertini's theorem, meets the smooth locus of $X$ transversely. Hence
\[
X\cap H\subseteq H\simeq \PP^{n-1}
\]
is a smooth hypersurface of degree $d$; see, for instance, \cite[III, Corollary~10.9]{Har}. Therefore the family of smooth hyperplane sections of $X$ defines a rational map
\begin{equation}
\label{eq:hyperplane-section-map}
\Phi(f):\PP^{n*}\dashrightarrow \M(d,n-1),
\qquad
H\longmapsto [X\cap H],
\end{equation}
where $\M(d,n-1)$ denotes the moduli space of smooth degree $d$ hypersurfaces in $\PP^{n-1}$.

The main problem is to understand when the hyperplane sections of $X$ vary maximally in moduli. We say that $\Phi(f)$ is generically finite onto its image if the general fiber of
\[
\Phi(f):\PP^{n*}\dashrightarrow \operatorname{Im}\Phi(f)
\]
is finite. Since $\dim \PP^{n*}=n$, this is equivalent to $\dim \operatorname{Im}\Phi(f)=n$. If $\dim \M(d,n-1)<n$, one cannot expect generic finiteness onto an $n$-dimensional image. In this case the appropriate replacement is maximal rank: we say that $\Phi(f)$ has maximal rank if
\[
\dim \operatorname{Im}\Phi(f)=\min\{n,\dim \M(d,n-1)\}.
\]

\begin{question}[Variation of hyperplane sections]
\label{q:variation-hyperplane-sections}
Let $X=V(f)\subseteq \PP^n$ be a degree $d$ hypersurface with at most isolated singularities. Determine algebraic and geometric conditions on $f$ under which the rational map $\Phi(f)$ is generically finite onto its image. More generally, determine when $\Phi(f)$ has maximal rank.
\end{question}

Question~\ref{q:variation-hyperplane-sections} is closely related to recent work of Beauville and of Patel--Riedl--Tseng. Beauville studied the case of a smooth cubic threefold $X\subseteq \PP^4$. In this case a hyperplane section is a cubic surface, and Beauville proved that a general smooth cubic surface is isomorphic to a hyperplane section of any smooth cubic threefold. His proof reduces the dominance of the corresponding moduli map to a Weak Lefschetz statement for the Jacobian ring of $X$; see \cite{B}. In the later work \cite{B2}, Beauville formulated the broader notion of maximal variation of a linear system and showed again that, for smooth hypersurfaces, the relevant infinitesimal condition is governed by multiplication by a general linear form on the Jacobian ring.

Patel--Riedl--Tseng studied the variation of linear slices of smooth hypersurfaces in projective space. They completely classified the plane curves whose line sections do not vary maximally in moduli, and in higher dimension they proved that the family of hyperplane sections of any smooth degree $d$ hypersurface in $\PP^n$ varies maximally when $d\geq n+3$; see \cite{PRT}. Their approach is geometric and uses vector bundle methods, including generalizations of the classical Grauert--Mülich theorem.

The point of the present paper is to place these results in a
Jacobian-algebraic framework which also applies to hypersurfaces with isolated
singularities. In the smooth case, our criterion recovers the Lefschetz
criterion appearing in the work of Beauville and, combined with recent weak
Lefschetz results, gives generic finiteness in the range $d\geq n+2$. For
singular hypersurfaces, however, a new obstruction appears: the existence of a
linear Jacobian syzygy. We show that, for non-cones, this is equivalent to the
existence of a positive-dimensional projective automorphism group. Thus the
failure of maximal infinitesimal variation is controlled by two explicit
algebraic mechanisms: the automorphism obstruction and the failure of the
critical Lefschetz map. The passage from this infinitesimal quotient to the
coarse moduli space is justified by a local GIT argument, using Luna's
\'etale slice theorem. This perspective also leads to new criteria using
singular hyperplane sections.

The aim of this paper is to give a Jacobian-algebraic and infinitesimal
approach to Question~\ref{q:variation-hyperplane-sections}. Instead of studying
the fibers of the global rational map
\[
\Phi(f):\PP^{n*}\dashrightarrow \M(d,n-1)
\]
directly, we compute the first-order variation of the hyperplane section modulo
infinitesimal coordinate changes on the hyperplane. This gives an explicit
infinitesimal quotient map. A local slice argument then allows us to pass from
the injectivity of this quotient map to the dimension of the image in the
coarse moduli space.

More precisely, let $J_f=(f_0,\ldots,f_n)$ be the Jacobian ideal of $f$, where $f_i=\partial f/\partial x_i$, and let $M(f)=S/J_f$ be the Milnor algebra. The
basic infinitesimal quotient is
\[
S_d/(J_f)_d=M(f)_d,
\]
which is the space of first-order degree $d$ deformations of $f$ modulo first-order projective changes of coordinates; see
Proposition~\ref{prop:tangent-orbit} and
Remark~\ref{rk:orbit-quotient-vs-moduli-tangent}. After choosing coordinates so that the general hyperplane is $H=V(x_0)$, and writing
$\overline h=h(0,x_1,\ldots,x_n)$, the first-order variation of the hyperplane section modulo coordinate changes on $H\simeq\PP^{n-1}$ is represented by
\[
\psi(f,x_0):\C^n\to R_d/(J_{\overline f})_d,\qquad
(b_1,\ldots,b_n)\mapsto
[(b_1x_1+\cdots+b_nx_n)\overline f_0],
\]
where $R=\C[x_1,\ldots,x_n]$. Thus the problem becomes to understand when $\psi(f,x_0)$ is injective at a general hyperplane.

Our first main result is the local criterion
Proposition~\ref{prop:local-criterion}. It shows that, if $X$ is not a cone, then $\psi(f,x_0)$ fails to be injective if and only if one of two obstructions
occurs: either $J_f$ has a linear syzygy, or multiplication by $x_0$ fails to be
injective in the critical degree,
\[
x_0:M(f)_{d-1}\to M(f)_d.
\]
Thus the failure of infinitesimal maximal variation has exactly two algebraic causes: a syzygetic obstruction and a Lefschetz obstruction; see also
Remark~\ref{rk:two-obstructions-mechanism}.

The first obstruction is geometric. If
$e=\operatorname{indeg}(\operatorname{Syz}(J_f))$, then Proposition~\ref{prop:aut-syz} proves that, for non-cones, $e=1$ is equivalent to the existence of a positive-dimensional projective automorphism group
$\Aut_{\PP^n}(X)$. This explains why automorphisms obstruct maximal variation: if $g\in\Aut_{\PP^n}(X)$, then $X\cap H$ and $X\cap g(H)$ are projectively
isomorphic. We also give a numerical way to exclude this obstruction. By a theorem of du Plessis--Wall, recalled in Theorem~\ref{thm:tjurina-bound}, the
inequality  $\tau(X)<(d-2)(d-1)^{n-1}$
implies $\dim\Aut_{\PP^n}(X)=0$; see Corollary~\ref{cor:no-aut-tau}.

Combining the local criterion with this interpretation gives our global criterion, Theorem~\ref{thm:main-criterion}. It states that if
$\dim\Aut_{\PP^n}(X)=0$ and, for a general linear form $\ell\in S_1$, the map
\[
\ell:M(f)_{d-1}\to M(f)_d
\]
is injective, then $\Phi(f)$ is generically finite onto its image. In the numerical ranges
\[
n=2,\ d\geq5;\qquad n=3,\ d\geq4;\qquad n\geq4,\ d\geq3,
\]
this injectivity follows from WLP in degree $d$. As a first consequence, using the weak Lefschetz theorem for equigenerated complete intersections due to
Beorchia--Mir\'o-Roig, we obtain Corollary~\ref{cor:smooth-high-degree}: if
$X\subseteq\PP^n$ is smooth, $n\geq3$, and $d\geq n+2$, then $\Phi(f)$ is generically finite onto its image. 
 This recovers \cite[Corollary~4.7]{BM} from the present infinitesimal criterion
and improves the range $d\geq n+3$ of \cite[Theorem~1.3]{PRT} by one degree. We then study the low-dimensional cases separately. For plane curves, we obtain a complete statement. Theorem~\ref{thm:plane-curves} shows that if
$C\subseteq\PP^2$ is a plane curve of degree $d\geq3$ with isolated singularities, then cones give constant maps; for non-cones, $\Phi(f)$ is
constant when $d=3$, has maximal rank $1$ when $d=4$, and for $d\geq5$ is generically finite onto its image if and only if $\dim\Aut_{\PP^2}(C)=0$. In particular, the Du Plessis--Wall bound for Tjurina 
$\tau(C)<(d-1)(d-2)$ guarantees generic finiteness. The proof uses the Lefschetz theorem of Dimca--Popescu for Milnor algebras of reduced plane curves~\cite[Corollary~4.4]{DPop}.
We also relate the exceptional plane curves to the Bourbaki degree in Remark~\ref{rk:plane-curve-bourbaki}.

For surfaces in $\PP^3$, the cubic case is exceptional because $\dim\M(3,2)=1$. Proposition~\ref{prop:cubic-surfaces} proves that a cubic surface with isolated singularities has maximal rank if it is not a cone, while a cubic cone gives a constant map. The proof is a direct first-order argument using Ilardi's injectivity theorem in low degree. For smooth surfaces, we obtain the sharp statement in Theorem~\ref{thm:smooth-surfaces}: smooth cubic surfaces have maximal rank but are not generically finite onto their image, while smooth
surfaces of degree $d\geq4$ give generically finite hyperplane-section maps. The quartic case is not covered by the general smooth high-degree corollary and uses the height-four complete-intersection Lefschetz theorem of Boij--Migliore--Mir\'o-Roig--Nagel.

The final section gives a different way to verify the Lefschetz condition in Theorem~\ref{thm:main-criterion}. Suppose that $X$ admits a hyperplane
$H'=V(\ell)$ avoiding $\Sing(X)$ such that the section
$W=X\cap H'\subseteq\PP^{n-1}$ is singular. Let $g$ be an equation of $W$, let
$r(g)=\operatorname{indeg}(\operatorname{Syz}(J_g))$, and let
$s(g)=\operatorname{indeg}(J_g^{\rm sat})$. A theorem of Dimca--Ilardi gives injectivity of
\[
\ell:M(f)_k\to M(f)_{k+1}
\]
for $k\leq\min\{d-3+r(g),d-3+s(g)\}$. Therefore, if $r(g)>1$ and $s(g)>1$, then
the critical map $\ell:M(f)_{d-1}\to M(f)_d$ is injective. This yields
Theorem~\ref{thm:singular-section-criterion}, a singular hyperplane-section criterion for generic finiteness of $\Phi(f)$.

Several consequences follow. Corollary~\ref{cor:nodal-section-criterion} shows that if $X$ satisfies the Tjurina bound excluding automorphisms and has a
hyperplane section with exactly $n$ nodes in general linear position, then $\Phi(f)$ is generically finite onto its image. In the surface case, this gives
Corollary~\ref{cor:singular-surfaces-three-nodes}: a surface in $\PP^3$ of degree $d\geq4$ with small Tjurina number and a plane section with three
noncollinear nodes has generically finite hyperplane-section map. Using the existence theorem of Dimca--Ilardi for nodal hyperplane sections of general
smooth hypersurfaces, we also obtain
Corollary~\ref{cor:general-smooth-hypersurfaces}: a general smooth hypersurface
of degree $d$ has generically finite hyperplane-section map whenever either
$d=3$ and $n\geq5$, or $d\geq4$ and $n\geq3$. Together with Beauville's cubic threefold theorem, this covers the general smooth cubic case in all dimensions
$n\geq4$. Finally, Example~\ref{ex:schoen-quintic} applies the singular-section criterion to the Schoen quintic threefold, showing that its hyperplane sections
vary generically finitely in moduli.


\section{The hyperplane-section map and infinitesimal quotients} \label{sec:map-and-quotients}

In this section we give the algebraic construction of the rational map $\Phi(f)$, recall the infinitesimal quotient by coordinate changes, and compute the
differential of $\Phi(f)$ on an affine chart.

\subsection{The hyperplane-section map}
\label{subsec:hyperplane-section-map}
Let $S=\C[x_0,\ldots,x_n]$ and let $f\in S_d$ be a homogeneous polynomial of degree
$d\geq 3$ defining the hypersurface $X=V(f)\subseteq \PP^n$. Assume that $X$ has at
most isolated singularities. Then a general hyperplane $H\subseteq \PP^n$ avoids
$\operatorname{Sing}(X)$ and, by Bertini's theorem, meets the smooth locus of $X$
transversely. Hence $X\cap H\subseteq H\simeq \PP^{n-1}$ is smooth. Equivalently,
\[
U_f=\{H\in \PP^{n*}\mid X\cap H \text{ is smooth}\}
\]
is a nonempty Zariski open subset of $\PP^{n*}$.

For a hyperplane $H=V(\ell)$, with $0\neq \ell\in S_1$, the section $X\cap H$ is
defined algebraically by the image of $f$ in $S/(\ell)$. Since $S/(\ell)$ is the
homogeneous coordinate ring of $H\simeq \PP^{n-1}$, the class of $f$ modulo $\ell$
defines a degree $d$ hypersurface in $H$. Thus, on $U_f$, we obtain a morphism
\[
\Phi(f):U_f\longrightarrow \M(d,n-1),\qquad H\longmapsto [X\cap H],
\]
and hence a rational map
\[
\Phi(f):\PP^{n*}\dashrightarrow \M(d,n-1).
\]

We now describe this map on an affine chart of $\PP^{n*}$. Consider the chart
where the coefficient of $x_0$ is nonzero. After normalizing this coefficient to be
$1$, every hyperplane in this chart has the form
\[
H=V(\ell),\qquad \ell=x_0-(a_1x_1+\cdots+a_nx_n),
\]
where $\mathbf a=(a_1,\ldots,a_n)\in \AAA^n$. Thus, on $H$, one has
\[
x_0=a_1x_1+\cdots+a_nx_n.
\]
Let $R=\C[x_1,\ldots,x_n]$ and $A=\C[a_1,\ldots,a_n]$. Restricting $f$ to the
hyperplanes in this chart is given by the $\C$-algebra homomorphism
\[
S\longrightarrow A\otimes_\C R
\]
defined by
\[
x_0\longmapsto a_1x_1+\cdots+a_nx_n,\qquad
x_i\longmapsto x_i\quad\text{for }1\leq i\leq n.
\]
Applying this homomorphism to $f$ gives
\[
F(\mathbf a,x)=f(a_1x_1+\cdots+a_nx_n,x_1,\ldots,x_n)
\in A\otimes_\C R_d.
\]
Thus we obtain a morphism
\[
\varphi_f:\AAA^n\longrightarrow R_d,\qquad
\mathbf a\longmapsto F(\mathbf a,x).
\]

Write $R_d=\bigoplus_{|\alpha|=d}\C x^\alpha$. Then
\[
F(\mathbf a,x)=\sum_{|\alpha|=d}F_\alpha(\mathbf a)x^\alpha
\]
for unique polynomials $F_\alpha(\mathbf a)\in A$. If $z_\alpha$ denotes the
coefficient coordinate on $R_d$ corresponding to $x^\alpha$, then
\[
\C[R_d]=\C[z_\alpha\mid |\alpha|=d],
\]
and $\varphi_f$ is given contravariantly by
\[
\varphi_f^\#:\C[z_\alpha\mid |\alpha|=d]\longrightarrow A,\qquad
z_\alpha\longmapsto F_\alpha(\mathbf a).
\]
Thus the coordinate functions of $\varphi_f$ are precisely the coefficients of
$F(\mathbf a,x)$.

Projectivization sends a nonzero form $G\in R_d$ to its class $[G]\in \PP(R_d)$.
Hence, after composing with $\varphi_f$ and restricting to the open subset where
$F(\mathbf a,x)$ defines a smooth hypersurface in $\PP^{n-1}$, we obtain a
rational map
\[
\overline{\varphi}_f:\AAA^n\dashrightarrow \PP(R_d)^{\rm sm},\qquad
\mathbf a\longmapsto [F(\mathbf a,x)].
\]
 The group $G=PGL_n(\C)$ acts on $\PP(R_d)$ by projective changes of coordinates
on $\PP^{n-1}$. Let $U_d\subseteq\PP(R_d)$ be the open subset parametrizing smooth
degree $d$ hypersurfaces. Since $d\geq3$, smooth hypersurfaces are stable for
this GIT action, and the quotient
\[
\pi:U_d\longrightarrow U_d/G
\]
is a geometric quotient with finite stabilizers. We denote this coarse moduli
space by $\M(d,n-1)$. Thus, on the affine chart of $\PP^{n*}$ considered above,
the hyperplane-section map is the composition
\[
\AAA^n\dashrightarrow U_d\longrightarrow \M(d,n-1),
\qquad
{\bf a}\longmapsto [F({\bf a},x)].
\]

We shall only use this quotient locally near points of $U_d$, through the local slice description recalled in Subsection~\ref{subsec:infinitesimal-hyperplane-map}.

\subsection{The infinitesimal quotient by coordinate changes}
\label{subsec:infinitesimal-quotient}
We now recall a standard orbit computation. It is the projective-linear version
of the usual singularity-theoretic description of tangent spaces to orbits by
vector fields; compare \cite[Chapter~4]{RCS}. We include the short proof for
convenience.

Let $G=GL_{n+1}(\C)$ act on $S_d$ by $g\cdot f=f\circ g^{-1}$.

\begin{prop}
\label{prop:tangent-orbit}
For every nonzero homogeneous form $f\in S_d$, one has
$T_fGf=(J_f)_d$. Consequently,
\[
E_f:=S_d/T_fGf=S_d/(J_f)_d=M(f)_d.
\]
\end{prop}

\begin{proof}
Let $\mathfrak g=\mathfrak{gl}_{n+1}(\C)$ be the Lie algebra of $G$. An element
$A=(a_{ij})\in\mathfrak g$ gives the first-order change of coordinates
\[
x_i\longmapsto x_i-\varepsilon\sum_{j=0}^n a_{ij}x_j,
\qquad \varepsilon^2=0.
\]
Thus
\[
f(x-\varepsilon Ax)
=
f(x)-\varepsilon\sum_{i,j}a_{ij}x_jf_i.
\]
Hence the tangent vectors to $Gf$ are precisely the linear combinations of the
forms $x_jf_i$. Therefore
\[
T_fGf=\langle x_jf_i\mid 0\leq i,j\leq n\rangle=(J_f)_d.
\]
Taking the quotient gives
$E_f=S_d/T_fGf=S_d/(J_f)_d=M(f)_d$.
\end{proof}
The equality $E_f=S_d/T_fGf=M(f)_d$ is purely algebraic and does not require $V(f)$ to be smooth. It means that
$M(f)_d$ is the space of first-order degree $d$ deformations of $f$ modulo first-order linear changes of coordinates. We also record the corresponding projective interpretation. The affine orbit $Gf\subseteq S_d$ is invariant under scalar  multiplication, and its image in $\PP(S_d)$ is the projective orbit. By Euler's formula,
$df=\sum_{i=0}^n x_if_i$, we have $f\in (J_f)_d=T_fGf$. Hence the scalar direction $\C f$ is already
contained in the tangent space to the affine orbit. Therefore passing from $S_d$ to $\PP(S_d)$ does not change the quotient:
\[
S_d/T_fGf
\cong
T_{[f]}\PP(S_d)/T_{[f]}(G'[f]),
\]
where $G'$ denotes the corresponding projective linear group.

\begin{rk}
\label{rk:orbit-quotient-vs-moduli-tangent}
The quotient
\[
E_f=S_d/T_fGf=M(f)_d
\]
should be understood as the infinitesimal quotient by coordinate changes. More
precisely, $S_d$ is the vector space of first-order degree $d$ deformations of
the equation $f$, while $T_fGf$ consists of those first-order deformations that
come only from linear changes of coordinates. Thus $E_f$ keeps only the
first-order deformations of $f$ which are nontrivial modulo coordinate changes.
Equivalently, $E_f$ is the normal space to the $G$-orbit of $f$ in $S_d$:
\[
E_f=\frac{\text{all first-order deformations of }f}
{\text{first-order deformations induced by coordinate changes}}.
\]
This interpretation is valid for every homogeneous form $f$, independently of
whether $V(f)$ is smooth or singular.

One should not automatically identify $E_f$ with the Zariski tangent space to a
coarse moduli space. If $V(f)$ is smooth and the projective stabilizer of $[f]$
is trivial, then the usual local-quotient computation identifies the tangent
space to the coarse moduli space at the point $p(f)$ represented by $f$ with the
corresponding infinitesimal quotient. More generally, when the stabilizer is
finite, Luna's \'{e}tale slice theorem identifies the infinitesimal quotient
with the tangent space to a local slice; the coarse moduli space is then locally
a finite quotient of this slice. Thus the infinitesimal quotient still computes
the dimension of images, although it need not be literally the Zariski tangent
space to the coarse quotient. This distinction is important in the presence of
nontrivial finite stabilizers or singularities of the quotient.
\end{rk}

\subsection{The infinitesimal form of the hyperplane-section map} \label{subsec:infinitesimal-hyperplane-map}
After a linear change of coordinates, we may assume that the chosen general
hyperplane is $H=V(x_0)\in U_f$.
We now compute the differential of the coefficient map at the point corresponding
to $H=V(x_0)$. For $h\in S$, write
$\overline h=h(0,x_1,\ldots,x_n)\in R$.
In particular, $\overline f=f(0,x_1,\ldots,x_n)$. Differentiating
\[
F(\mathbf a,x)=f(a_1x_1+\cdots+a_nx_n,x_1,\ldots,x_n)
\]
with respect to $a_j$ and evaluating at $\mathbf a=0$, we get
$\frac{\partial F}{\partial a_j}(0,x)=x_j\overline f_0$.
Hence
\[
d_0\varphi_f:\C^n\longrightarrow R_d
\]
is given by
\[
(b_1,\ldots,b_n)\longmapsto
(b_1x_1+\cdots+b_nx_n)\overline f_0.
\]

To pass from equations to moduli, we quotient by infinitesimal changes of
coordinates on $H\simeq \PP^{n-1}$. Applying Proposition~\ref{prop:tangent-orbit}
to the form $\overline f\in R_d$, with the group $GL_n(\C)$ acting on
$R=\C[x_1,\ldots,x_n]$, gives
\[
T_{\overline f}(GL_n(\C)\overline f)=J(\overline f)_d.
\]
Therefore, after quotienting by infinitesimal coordinate changes on
$H\simeq\PP^{n-1}$, we are led to the map
\begin{equation}
\label{eq:psi-local}
\psi(f,x_0):\C^n\to R_d/(J_{\overline f})_d,\qquad
(b_1,\ldots,b_n)\mapsto
\left[(b_1x_1+\cdots+b_nx_n)\overline f_0\right].
\end{equation}
The following lemma makes precise the sense in which this quotient map is the
infinitesimal differential of the hyperplane-section map to moduli.

We shall use one standard local fact about the GIT quotient. Smooth degree $d$
hypersurfaces in $\PP^{n-1}$, with $d\geq3$, are stable for the natural action
of $G=PGL_n(\C)$; see \cite[Chapter~4, \S2]{MFK}. In particular, their
projective stabilizers are finite. By Luna's \'{e}tale slice theorem, the
quotient is locally described by a transverse slice modulo this finite
stabilizer; see
\cite[Chapitre~III, \S1, Th\'eor\`eme du slice \'etale, p.~97]{Luna}.
Moreover, in the smooth case the slice may be chosen smooth and tangent to the
normal space to the orbit; see
\cite[Chapitre~III, \S1, Remarque~1$^\circ$ following the theorem, p.~97]{Luna}.
We use this only to pass from the infinitesimal quotient by coordinate changes
to the dimension of the image in the coarse moduli space.

\begin{lem}
\label{lem:psi-luna-slice}
Let $R=\C[y_1,\ldots,y_n]$, let $U_d\subseteq\PP(R_d)$ be the open subset
parametrizing smooth degree $d$ hypersurfaces in $\PP^{n-1}$, and let
$G=PGL_n(\C)$. Let
\[
\pi:U_d\longrightarrow \M(d,n-1)
\]
be the coarse GIT quotient by projective coordinate changes. Let $B$ be a smooth
variety and let $\alpha:B\to U_d$ be a morphism. Fix $b\in B$, and write
$\alpha(b)=[g]$, where $g\in R_d$ defines a smooth hypersurface. Then the
infinitesimal quotient map attached to $\pi\circ\alpha$ at $b$ is the composite
\[
T_bB\longrightarrow T_{[g]}U_d
\longrightarrow
T_{[g]}U_d/T_{[g]}(G\cdot[g])
\simeq R_d/(J_g)_d .
\]
Equivalently, if a local homogeneous lift of $\alpha$ near $b$ is represented
by a family of forms $G_B$, with $G_B(b)=g$, then this map sends
$v\in T_bB$ to the class of $v(G_B)$ in $R_d/(J_g)_d$.

If this infinitesimal quotient map is injective, then the image of
$\pi\circ\alpha$ has dimension $\dim B$ near $\pi([g])$.
\end{lem}

\begin{proof}
Since $U_d$ is open in $\PP(R_d)$, one has
$T_{[g]}U_d=R_d/\C g$. The tangent space to the $G$-orbit is obtained from
infinitesimal projective coordinate changes. Equivalently, if we compute first
on the affine cone $R_d$, the tangent space to the $GL_n(\C)$-orbit of $g$ is
generated by the forms $y_jg_i$, $1\leq i,j\leq n$. Hence
\[
T_g(GL_n(\C)\cdot g)=\langle y_jg_i\mid 1\leq i,j\leq n\rangle=(J_g)_d .
\]
By Euler's formula, $\deg(g)g=\sum_{i=1}^n y_ig_i$, so $g\in (J_g)_d$.
Therefore, after passing to projective space,
\[
T_{[g]}(G\cdot[g])
=
\frac{(J_g)_d}{\C g}
\subseteq
\frac{R_d}{\C g}
=
T_{[g]}U_d .
\]
It follows that the normal space to the orbit is
\[
T_{[g]}U_d/T_{[g]}(G\cdot[g])
\simeq
\frac{R_d/\C g}{(J_g)_d/\C g}
\simeq
R_d/(J_g)_d .
\]

Let $G_B$ be a local homogeneous lift of $\alpha$ near $b$, with $G_B(b)=g$.
Then $d_b\alpha(v)$ is represented by $v(G_B)$ modulo $\C g$. Quotienting
further by infinitesimal coordinate changes gives precisely the class of
$v(G_B)$ in $R_d/(J_g)_d$. This proves the asserted description of the
infinitesimal quotient map.

It remains to justify the final dimension statement. Since $g$ is smooth and
$d\geq3$, the point $[g]\in U_d$ is stable for the action of $G=PGL_n(\C)$, and
its stabilizer
\[
\Gamma=\operatorname{Stab}_G([g])
=\Aut_{\PP^{n-1}}(V(g))
\]
is finite. By Luna's \'{e}tale slice theorem
\cite[Chapitre~III, \S1, Th\'eor\`eme du slice \'etale, p.~97]{Luna}, together
with the smooth-slice refinement
\cite[Chapitre~III, \S1, Remarque~1$^\circ$ following the theorem, p.~97]{Luna},
there is an \'{e}tale local model of the quotient near $\pi([g])$ of the form
$N/\Gamma$, where $N$ is a smooth transverse slice through $[g]$ and
\[
T_{[g]}N\simeq T_{[g]}U_d/T_{[g]}(G\cdot[g])
\simeq R_d/(J_g)_d .
\]
Equivalently, after passing to this \'{e}tale local model, the normal component
of the differential of $\alpha$ is exactly the infinitesimal quotient map
described above. If this map is injective, then the image in the smooth slice
$N$ has dimension $\dim B$ near $[g]$. Since $N\to N/\Gamma$ is finite, passing
to the coarse quotient does not change the dimension of the image. Hence the
image of $\pi\circ\alpha$ has dimension $\dim B$ near $\pi([g])$.
\end{proof}

Applying the lemma to the hyperplane-section family, the local homogeneous lift
is
\[
F({\bf a},x)=f(a_1x_1+\cdots+a_nx_n,x_1,\ldots,x_n),
\]
and differentiating at ${\bf a}=0$ gives
\[
d_0F(b_1,\ldots,b_n)
=
(b_1x_1+\cdots+b_nx_n)\overline f_0.
\]
Thus the infinitesimal quotient differential of the hyperplane-section map at
$H=V(x_0)$ is precisely the map $\psi(f,x_0)$ in \eqref{eq:psi-local}.

\subsection{Jacobian syzygies and projective automorphisms}
Let $J_f=(f_0,\ldots,f_n)\subseteq S$ be the Jacobian ideal of $f\in S_d$. We denote by 
$\operatorname{Syz}(J_f)$ the module of first syzygies of $J_f$, with degrees measured in the standard
grading of $S$. It fits into the shifted graded presentation
\[
0\to \operatorname{Syz}(J_f)\to S^{n+1}\to J_f(d-1)\to 0
\]
Let
$
\operatorname{Der}_{\C}(S)=
\bigoplus_{i=0}^n S\frac{\partial}{\partial x_i}
$
be the $S$-module of $\C$-derivations of $S$, and set
\[
\operatorname{Der}_f(S)^0
=
\{\Theta\in \operatorname{Der}_{\C}(S)\mid \Theta(f)=0\}.
\]
Then
\[
\operatorname{Syz}(J_f)\simeq \operatorname{Der}_f(S)^0,
\qquad
(a_0,\ldots,a_n)\longmapsto
\sum_{i=0}^n a_i\frac{\partial}{\partial x_i},
\]
where the grading on $\operatorname{Der}_f(S)^0$ is given by the degrees of the
coefficients $a_i$. We set
\[
e=\operatorname{indeg}(\operatorname{Syz}(J_f)).
\]
Thus $e=1$ means that $f_0,\ldots,f_n$ admit a nontrivial linear syzygy, or
equivalently that $f$ is killed by a nonzero linear vector field.
We assume that $X=V(f)$ is not a cone, i.e., after a linear change of coordinates,
$f$ depends on all variables. Thus 
$\operatorname{Syz}(J_f)_0=0$ and hence $e\geq 1$.

Let $G=GL_{n+1}(\C)$ and $G'=PGL_{n+1}(\C)$. We denote by
\[
\Aut_{\PP^n}(X)
=
\{g\in PGL_{n+1}(\C)\mid g(X)=X\}
\]
the group of projective automorphisms of $X$. Equivalently, if $g\in G$
represents a projective transformation, then $g$ preserves $X=V(f)$ if and only if
\[
g\cdot f=\lambda f
\]
for some $\lambda\in \C^*$. Indeed, $f$ and $\lambda f$ define the same
hypersurface in projective space. Thus $\Aut_{\PP^n}(X)$ is the stabilizer of
$[f]\in \PP(S_d)$ under the action of $G'$.

We say that $X$ has a positive-dimensional projective automorphism group if
\[
\dim \Aut_{\PP^n}(X)>0.
\]

\begin{prop}
\label{prop:aut-syz}
Assume that $X=V(f)\subseteq \PP^n$ is not a cone. Then the following are equivalent:
\begin{enumerate}
\item $e=1$.
\item The forms $x_jf_i$, $0\leq i,j\leq n$, are linearly dependent in $S_d$.
\item $\dim (J_f)_d<(n+1)^2$.
\item $\dim M(f)_d>\binom{n+d}{n}-(n+1)^2$.
\item $\dim \Aut_{\PP^n}(X)>0$.
\end{enumerate}
\end{prop}

\begin{proof}
Since $X$ is not a cone, there is no nonzero constant syzygy among
$f_0,\ldots,f_n$. Indeed, a constant syzygy
$c_0f_0+\cdots+c_nf_n=0$ gives a nonzero constant vector field killing $f$;
after a linear change of coordinates this vector field becomes
$\partial/\partial y_0$, and hence $f$ is independent of $y_0$, so $X$ is a
cone. Conversely, if $X$ is a cone, then after a linear change of coordinates
$f$ is independent of one variable, giving a constant syzygy. Thus
$\operatorname{Syz}(J_f)_0=0$, and hence $e\geq1$.
Thus $e=1$ is equivalent to the existence of a nonzero linear syzygy
\[
L_0f_0+\cdots+L_nf_n=0,\qquad L_i\in S_1.
\]
Writing $L_i=\sum_j a_{ij}x_j$, this is precisely a nontrivial linear relation
among the forms $x_jf_i$. Hence (1) and (2) are equivalent.

Since
\[
(J_f)_d=\langle x_jf_i\mid 0\leq i,j\leq n\rangle,
\]
condition (2) is equivalent to $\dim (J_f)_d<(n+1)^2$. This proves the equivalence of (2) and (3). Also,
\[
\dim M(f)_d=\dim S_d-\dim (J_f)_d
=\binom{n+d}{n}-\dim (J_f)_d,
\]
so condition (3) is equivalent to condition (4).
It remains to compare with projective automorphisms. The Lie algebra of the
stabilizer of $[f]\in \PP(S_d)$ consists of matrices $A\in\mathfrak{gl}_{n+1}(\C)$
such that $\nabla f\cdot Ax\in \C f$. If $\dim\Aut_{\PP^n}(X)>0$, then there is
such an $A$ with nonzero class in $\mathfrak{pgl}_{n+1}(\C)$ and
\[
\nabla f\cdot Ax=\lambda f
\]
for some $\lambda\in\C$. By Euler's formula, $df=\sum_i x_if_i$. Hence, replacing
$A$ by $A'=A-(\lambda/d)I$, we get
\[
\nabla f\cdot A'x=0,
\]
with $A'\neq 0$ in $\mathfrak{gl}_{n+1}(\C)$. Thus the forms $x_jf_i$ are linearly
dependent, so $(5)$ implies $(2)$. 

Conversely, suppose that the forms $x_jf_i$ are linearly dependent. Then there
exists a nonzero matrix $A=(a_{ij})\in\mathfrak{gl}_{n+1}(\C)$ such that
\[
\nabla f\cdot Ax=0.
\]
The matrix $A$ is not scalar: if $A=\lambda I$, then
$\nabla f\cdot Ax=\lambda\sum_i x_if_i=\lambda d f$, hence $\lambda=0$ since
$f\neq0$. Thus $A$ has nonzero image in $\mathfrak{pgl}_{n+1}(\C)$. The
corresponding one-parameter subgroup $\exp(tA)$ satisfies
\[
\frac{d}{dt}f(\exp(tA)x)=\nabla f(\exp(tA)x)\cdot A\exp(tA)x=0,
\]
and therefore preserves $f$. Hence it gives a nontrivial one-parameter family
in $\Aut_{\PP^n}(X)$, so $\dim\Aut_{\PP^n}(X)>0$.
\end{proof}

Geometrically, the condition $e=1$ gives the first obstruction to the immersion
of $\Phi(f)$. Indeed, by Proposition~\ref{prop:aut-syz}, this condition is
equivalent to $\dim\Aut_{\PP^n}(X)>0$. If $g\in\Aut_{\PP^n}(X)$ and $H$ is a
hyperplane, then
\[
X\cap g(H)=g(X\cap H).
\]
Thus $X\cap H$ and $X\cap g(H)$ are projectively isomorphic, and hence define
the same point of $\M(d,n-1)$. Consequently, $\Phi(f)$ is constant along the
$\Aut_{\PP^n}(X)$-orbit of $H$. At a general hyperplane, this gives nonzero
tangent directions killed by $d\Phi(f)$.

\begin{rk}
\label{rk:du-plessis-wall-symmetry}
For plane curves, Proposition~\ref{prop:aut-syz} is closely related to the criterion of du Plessis--Wall: a reduced plane curve has a one-dimensional
projective symmetry if and only if its defining equation is killed by a nonzero linear vector field; see \cite[Proposition~1.1]{duPCTC2}. Their
\cite[Proposition~1.2]{duPCTC2} further shows that, for reduced plane curves, a two-dimensional symmetry group can occur only in degree $d=3$. Thus, in the
plane case and for $d>3$, the obstruction $e=1$ corresponds precisely to the existence of a one-dimensional projective symmetry.
\end{rk}

\subsection{A numerical way to exclude automorphisms}

For hypersurfaces with isolated singularities, the automorphism obstruction can
often be excluded numerically. For $p\in\Sing(X)$, choose local coordinates
$y_1,\ldots,y_n$ and a local equation $g\in\mathcal O_{\PP^n,p}$ of $X$ at $p$.
The local Tjurina number is
\[
\tau(X,p)=\dim_\C
\frac{\mathcal O_{\PP^n,p}}{(g,\partial g/\partial y_1,\ldots,\partial g/\partial y_n)}.
\]
We set $\tau(X)=\sum_{p\in\Sing(X)}\tau(X,p)$.

\begin{thm}[{\cite[Theorem~5.3]{duPCTC01}}]
\label{thm:tjurina-bound}
Let $X=V(f)\subseteq \PP^n$ be a degree $d$ hypersurface with at most isolated
singularities, and put $e=\operatorname{indeg}(\operatorname{Syz}(J_f))$. Then
\[
(d-e-1)(d-1)^{n-1}\leq \tau(X)
\leq (d-1)^n-e(d-e-1)(d-1)^{n-2}.
\]
\end{thm}

\begin{cor}
\label{cor:no-aut-tau}
If $X=V(f)$ has at most isolated singularities and
$\tau(X)<(d-2)(d-1)^{n-1}$, then $\dim\Aut_{\PP^n}(X)=0$.
\end{cor}

\begin{proof}
Let $e=\operatorname{indeg}(\operatorname{Syz}(J_f))$. If $e=0$, then there is a
nonzero constant syzygy among $f_0,\ldots,f_n$. As explained in the proof of
Proposition~\ref{prop:aut-syz}, this means that $X$ is a cone. In this case
Theorem~\ref{thm:tjurina-bound} gives
$\tau(X)\geq (d-1)^n$,
which contradicts the hypothesis
$\tau(X)<(d-2)(d-1)^{n-1}$. Hence $e\geq1$, and $X$ is not a cone.
If $\dim\Aut_{\PP^n}(X)>0$, then Proposition~\ref{prop:aut-syz} gives $e=1$.
Therefore Theorem~\ref{thm:tjurina-bound} gives
\[
\tau(X)\geq (d-2)(d-1)^{n-1},
\]
again contradicting the hypothesis. Hence
$\dim\Aut_{\PP^n}(X)=0$.
\end{proof}

\section{The immersion criterion}
\label{sec:immersion-criterion}

We now explain the infinitesimal strategy used to study Question~\ref{q:variation-hyperplane-sections}. The domain of the hyperplane-section map is $\PP^{n*}$, which has dimension $n$. Hence $\Phi(f)$
has maximal variation, i.e. is generically finite onto its image, precisely when $\dim\operatorname{Im}\Phi(f)=n$. One way to prove this is to show that the first-order variation of $X\cap H$ in moduli is injective at a general
hyperplane $H$.

On the affine chart considered in Subsection~\ref{subsec:infinitesimal-hyperplane-map},
the map $\varphi_f$ records the equation of the hyperplane section. Its differential measures the first-order change of this equation. However, two
first-order changes may define the same point of moduli if they differ by an infinitesimal coordinate change on the hyperplane. Thus the relevant object is
not $d\varphi_f$ itself, but its image modulo the tangent space to the coordinate-change orbit. This is the map $\psi(f,x_0)$ defined in \eqref{eq:psi-local}.

The next proposition is the local algebraic form of Beauville's Lefschetz
criterion for hyperplane sections; compare \cite[Proposition]{B} and
\cite[Proposition~2]{B2}. Compared with the smooth setting, the isolated
singular case has one extra infinitesimal obstruction, namely a linear Jacobian
syzygy.

\begin{prop}
\label{prop:local-criterion}
Assume that $X=V(f)\subseteq \PP^n$ is not a cone and that $H_0=V(x_0)\in U_f$. Then
$\psi(f,x_0)$ is not injective if and only if one of the following conditions
holds:
\begin{enumerate}
\item $e=\operatorname{indeg}(\operatorname{Syz}(J_f))=1$;
\item the multiplication map $x_0:M(f)_{d-1}\to M(f)_d$ is not injective.
\end{enumerate}
\end{prop}

\begin{proof}
Assume first that $\psi(f,x_0)$ is not injective. Then there is a nonzero
linear form $L=b_1x_1+\cdots+b_nx_n\in R_1$ such that
$L\overline f_0\in (J_{\overline f})_d$. Since
$J_{\overline f}=(\overline f_1,\ldots,\overline f_n)$, there are
$\ell_1,\ldots,\ell_n\in R_1$ with
\[
L\overline f_0=\ell_1\overline f_1+\cdots+\ell_n\overline f_n.
\]
Viewing $L,\ell_1,\ldots,\ell_n$ as linear forms in $S$ not involving $x_0$,
the form $Lf_0-\sum_{i=1}^n\ell_if_i$ restricts to zero on $x_0=0$. Hence
\[
Lf_0-\sum_{i=1}^n\ell_if_i=x_0h
\]
for some $h\in S_{d-1}$.

If $h\in (J_f)_{d-1}$, then $h=c_0f_0+\cdots+c_nf_n$ with $c_i\in\C$.
Substitution gives
\[
(L-c_0x_0)f_0+\sum_{i=1}^n(-\ell_i-c_ix_0)f_i=0.
\]
This is a nontrivial linear syzygy, since $L-c_0x_0\neq0$. As $X$ is not a
cone, $e\geq1$, and hence $e=1$. If $h\notin (J_f)_{d-1}$, then
$[h]\neq0$ in $M(f)_{d-1}$, while $x_0h\in (J_f)_d$. Thus $x_0[h]=0$ in
$M(f)_d$, so $x_0:M(f)_{d-1}\to M(f)_d$ is not injective.

Conversely, assume first that $e=1$. Then there is a nontrivial linear syzygy
\[
A_0f_0+\cdots+A_nf_n=0,\qquad A_i\in S_1.
\]
Reducing modulo $x_0$, we get
\[
\overline A_0\,\overline f_0+
\overline A_1\,\overline f_1+\cdots+\overline A_n\,\overline f_n=0.
\]
Since $V(\overline f)\subseteq\PP^{n-1}$ is smooth,
$\overline f_1,\ldots,\overline f_n$ form a regular sequence in $R$. Hence their
first syzygies are Koszul, and they have no linear syzygy because $d\geq3$.

We claim that $\overline A_0\neq0$. If $\overline A_0=0$, then
$\overline A_1\overline f_1+\cdots+\overline A_n\overline f_n=0$, so
$\overline A_i=0$ for all $i\geq1$. Hence $A_i=c_ix_0$ for some $c_i\in\C$,
and the syzygy becomes $x_0(c_0f_0+\cdots+c_nf_n)=0$. Thus
$c_0f_0+\cdots+c_nf_n=0$, giving a constant syzygy among the partial
derivatives. This would make $X$ a cone, a contradiction. Hence
$\overline A_0\neq0$.

Setting $L=\overline A_0\in R_1\setminus\{0\}$, the reduced syzygy gives
$L\overline f_0\in (J_{\overline f})_d$. Thus $L$ gives a nonzero element of
$\ker\psi(f,x_0)$, and $\psi(f,x_0)$ is not injective.

Finally, assume that $x_0:M(f)_{d-1}\to M(f)_d$ is not injective. Then there is
$h\in S_{d-1}$ with $h\notin (J_f)_{d-1}$ and $x_0h\in (J_f)_d$. Write
\[
x_0h=A_0f_0+\cdots+A_nf_n,\qquad A_i\in S_1.
\]
Reducing modulo $x_0$, we get
\[
\overline A_0\,\overline f_0+
\overline A_1\,\overline f_1+\cdots+\overline A_n\,\overline f_n=0.
\]
If $\overline A_0=0$, then
$\overline A_1\overline f_1+\cdots+\overline A_n\overline f_n=0$. Since
$\overline f_1,\ldots,\overline f_n$ have no linear syzygy, all
$\overline A_i$ vanish. Hence $A_i=c_ix_0$ for all $i$, and cancelling $x_0$
gives $h=c_0f_0+\cdots+c_nf_n\in (J_f)_{d-1}$, a contradiction. Therefore
$\overline A_0\neq0$. Setting $L=\overline A_0$, we get
$L\overline f_0\in (J_{\overline f})_d$, so $\psi(f,x_0)$ is not injective.
\end{proof}

\begin{rk}
\label{rk:two-obstructions-mechanism}
The proof of Proposition~\ref{prop:local-criterion} explains the origin of the
two obstructions. A nonzero element of $\ker\psi(f,x_0)$ gives a relation
$L\overline f_0\in (J_{\overline f})_d$ for some $0\neq L\in R_1$. Lifting this
relation to $S$ produces an identity
\[
Lf_0-\sum_{i=1}^n\ell_i f_i=x_0h
\]
with $h\in S_{d-1}$. If $h\in (J_f)_{d-1}$, then the identity gives a linear
Jacobian syzygy, hence the automorphism obstruction. If
$h\notin (J_f)_{d-1}$, then the class of $h$ is a nonzero element of
$M(f)_{d-1}$ killed by multiplication by $x_0$, hence the Lefschetz obstruction.
Thus the failure of infinitesimal injectivity is accounted for exactly by these
two mechanisms.
\end{rk}

Let $M=\bigoplus_{k\geq0}M_k$ be a finitely generated graded $S$-module. We say
that $M$ satisfies the Weak Lefschetz Property in degree $j$ if, for a general
linear form $\ell\in S_1$, the multiplication map $\ell:M_{j-1}\to M_j$ has
maximal rank. By upper semicontinuity of rank, this is equivalent to the
existence of one linear form $\ell\in S_1$ for which $\ell:M_{j-1}\to M_j$ has
maximal rank.

We apply this to $M=M(f)$ and $j=d$. Assume that
$\dim\Aut_{\PP^n}(X)=0$. Then $X$ is not a cone and, by
Proposition~\ref{prop:aut-syz}, one has $e>1$. Hence there are no syzygies among
$f_0,\ldots,f_n$ in degrees $0$ and $1$, and therefore
\[
\dim M(f)_{d-1}=\dim S_{d-1}-(n+1),\qquad
\dim M(f)_d=\dim S_d-(n+1)^2.
\]
Thus
\[
\dim M(f)_d-\dim M(f)_{d-1}
=\binom{n+d-1}{n-1}-n(n+1).
\]
This number is nonnegative for
\[
n=2,\ d\geq5;\qquad n=3,\ d\geq4;\qquad n\geq4,\ d\geq3.
\]
In these ranges, maximal rank of $\ell:M(f)_{d-1}\to M(f)_d$ is equivalent to
injectivity. Hence WLP in degree $d$ gives precisely the Lefschetz condition
needed in Proposition~\ref{prop:local-criterion}.

The local criterion now yields the global form of the main result, giving a Jacobian-algebraic answer to Question~\ref{q:variation-hyperplane-sections}. It extends Beauville's Lefschetz criterion for smooth hypersurfaces
\cite[Proposition]{B}, \cite[Proposition~2]{B2} to hypersurfaces with isolated singularities: besides the Lefschetz obstruction, the new obstruction is the presence of a positive-dimensional projective automorphism group. Thus maximal variation is obtained by excluding these two obstructions. This gives a criterion complementary to the global methods of Patel--Riedl--Tseng~\cite{PRT}.
\begin{thm}
\label{thm:main-criterion}
Let $X=V(f)\subseteq\PP^n$ be a degree $d$ hypersurface with at most isolated
singularities. Assume that $\dim\Aut_{\PP^n}(X)=0$ and that, for a general
linear form $\ell\in S_1$, the multiplication map
$\ell:M(f)_{d-1}\to M(f)_d$ is injective. Then
\[
\Phi(f):\PP^{n*}\dashrightarrow \M(d,n-1)
\]
is generically finite onto its image.

In particular, if $\tau(X)<(d-2)(d-1)^{n-1}$, if $(n,d)$ lies in one of the
ranges
\[
n=2,\ d\geq5;\qquad n=3,\ d\geq4;\qquad n\geq4,\ d\geq3,
\]
and if $M(f)$ satisfies WLP in degree $d$, then $\Phi(f)$ is generically finite
onto its image.
\end{thm}

\begin{proof}
Choose a general linear form $\ell\in S_1$ such that $H=V(\ell)\in U_f$ and
such that multiplication by $\ell$ is injective:
\[
\ell:M(f)_{d-1}\to M(f)_d.
\]
After a projective change of coordinates, we may assume that $\ell=x_0$ and
$H=V(x_0)$. Set $R=\C[x_1,\ldots,x_n]$ and
\[
\overline f=f(0,x_1,\ldots,x_n)\in R_d .
\]
Since $H\in U_f$, the hypersurface $V(\overline f)\subseteq\PP^{n-1}$ is
smooth.

By assumption, $\dim\Aut_{\PP^n}(X)=0$. Hence $X$ is not a cone, and by
Proposition~\ref{prop:aut-syz} one has
\[
\operatorname{indeg}(\operatorname{Syz}(J_f))>1.
\]
Thus the linear-syzygy obstruction in Proposition~\ref{prop:local-criterion}
does not occur. By the choice of $x_0$, the Lefschetz obstruction also does not
occur. Therefore Proposition~\ref{prop:local-criterion} gives that
\[
\psi(f,x_0):\C^n\to R_d/(J_{\overline f})_d
\]
is injective.

On the affine chart of $\PP^{n*}$ centered at $H=V(x_0)$, the hyperplane-section
family is given by
\[
F({\bf a},x)=f(a_1x_1+\cdots+a_nx_n,x_1,\ldots,x_n).
\]
Differentiating at ${\bf a}=0$ in the direction
$b=(b_1,\ldots,b_n)$ gives
\[
d_0F(b)=(b_1x_1+\cdots+b_nx_n)\overline f_0.
\]
Thus, by Lemma~\ref{lem:psi-luna-slice}, the induced infinitesimal quotient map
of the hyperplane-section family at $H$ is precisely $\psi(f,x_0)$.

Since $\psi(f,x_0)$ is injective, Lemma~\ref{lem:psi-luna-slice} implies that
the image of $\Phi(f)$ has dimension $n$ near $\Phi(f)(H)$. Hence
\[
\dim\operatorname{Im}\Phi(f)=n.
\]
Because $\dim\PP^{n*}=n$, the general fiber of
\[
\Phi(f):\PP^{n*}\dashrightarrow \operatorname{Im}\Phi(f)
\]
is finite. Therefore $\Phi(f)$ is generically finite onto its image.

For the final assertion, the inequality
$\tau(X)<(d-2)(d-1)^{n-1}$ excludes the automorphism obstruction by
Corollary~\ref{cor:no-aut-tau}. In the stated numerical ranges, maximal rank of
\[
\ell:M(f)_{d-1}\to M(f)_d
\]
is equivalent to injectivity. Hence, if $M(f)$ satisfies WLP in degree $d$,
multiplication by a general linear form is injective in degree $d-1$. The first
part applies.
\end{proof}

We now recover the smooth case. Theorem~\ref{thm:main-criterion} gives a positive answer to \cite[Question~1.1]{PRT} whenever the required Lefschetz multiplication map is injective in degree $d$. In the smooth case the Jacobian ideal is a complete intersection, and \cite[Theorem~4.5]{BM} shows that the Jacobian algebra satisfies WLP in every degree $t<d-1+\left\lceil\frac{d-1}{n}\right\rceil$. In particular, WLP holds in degree $d$ whenever $d\geq n+2$. Thus Theorem~\ref{thm:main-criterion} recovers
\cite[Corollary~4.7]{BM} from the present  infinitesimal criterion and improves the range $d\geq n+3$ of \cite[Theorem~1.3]{PRT} by one degree.

\begin{cor}
\label{cor:smooth-high-degree}
Let $X=V(f)\subseteq\PP^n$ be smooth, with $n\geq3$ and $d\geq n+2$. Then
\[
\Phi(f):\PP^{n*}\dashrightarrow \M(d,n-1)
\]
is generically finite onto its image.
\end{cor}
\begin{proof}
Since $X$ is smooth, $\tau(X)=0$, so Corollary~\ref{cor:no-aut-tau} excludes the
automorphism obstruction. By \cite[Theorem~4.5]{BM}, $M(f)$ satisfies WLP in
degree $d$ for $d\geq n+2$. Since $n\geq3$ and $d\geq n+2$, the pair $(n,d)$ lies
in the numerical ranges of Theorem~\ref{thm:main-criterion}. Hence
Theorem~\ref{thm:main-criterion} applies.
\end{proof}
\section{Plane curves}
\label{sec:plane-curves}

In this section we specialize the preceding criterion to plane curves. Thus
$n=2$, $S=\C[x_0,x_1,x_2]$, and $C=V(f)\subseteq\PP^2$ is a plane curve of
degree $d\geq3$ with isolated singularities. In particular, $C$ is reduced. The
hyperplane-section map is
\[
\Phi(f):\PP^{2*}\dashrightarrow \M(d,1),
\]
and sends a general line $L$ to the moduli class of the reduced divisor
$C\cap L\subseteq L\simeq\PP^1$.

A smooth degree $d$ hypersurface in $\PP^1$ is an unordered set of $d$ distinct
points. Hence $\M(d,1)$ is the quotient of a nonempty open subset of
$\PP(R_d)\simeq\PP^d$ by $PGL_2(\C)$. Since $\dim PGL_2=3$ and the general
stabilizer is finite for $d\geq3$, one has $\dim\M(d,1)=d-3$. Thus maximal rank
means rank $0$ for $d=3$, rank $1$ for $d=4$, and rank $2$ for $d\geq5$.

\begin{lem}[{\cite[Corollary~4.4]{DPop}}]
\label{lem:DP-plane-lefschetz}
Let $C=V(f)\subseteq\PP^2$ be a reduced plane curve of degree $d$, and let
$M(f)=S/J_f$ be its Milnor algebra. Then, for a general linear form
$\ell\in S_1$, the multiplication map
\[
\ell:M(f)_k\to M(f)_{k+1}
\]
is injective for every $k<3(d-2)/2$.
\end{lem}

\begin{thm}
\label{thm:plane-curves}
Let $C=V(f)\subseteq\PP^2$ be a plane curve of degree $d\geq3$ with isolated
singularities. If $C$ is a cone, then $\Phi(f)$ is constant. Assume now that $C$
is not a cone. Then:
\begin{enumerate}
\item if $d=3$, the map $\Phi(f)$ is constant, hence has maximal rank but is not
generically finite onto its image;
\item if $d=4$, the map $\Phi(f)$ has maximal rank $1$, hence is not generically
finite onto its image;
\item if $d\geq5$, then $\Phi(f)$ is generically finite onto its image if and
only if $\dim\Aut_{\PP^2}(C)=0$.
\end{enumerate}
In particular, if $d\geq5$ and $\tau(C)<(d-1)(d-2)$, then $\Phi(f)$ is
generically finite onto its image.
\end{thm}

\begin{proof}
If $C$ is a cone, then $C$ is the cone over a reduced divisor of degree $d$ on
$\PP^1$. Projection from the vertex identifies every general line with the base
line and sends $C\cap L$ to this fixed divisor. Hence all general line sections
are projectively equivalent, and $\Phi(f)$ is constant. We may therefore assume
that $C$ is not a cone.

Let $H=V(\ell)$ be a general line. Whenever a Lefschetz multiplication map is
used, we choose $\ell$ in the corresponding nonempty open subset of $S_1$. After
a projective linear change of coordinates, we may assume that $\ell=x_0$, so
$H=V(x_0)$. We keep the notation $f$ for the transformed equation. Set
$R=\C[x_1,x_2]$ and write $\overline h=h(0,x_1,x_2)$ for $h\in S$. Since
$H\in U_f$, the binary form $\overline f$ has $d$ distinct roots. Therefore
$J_{\overline f}=(\overline f_1,\overline f_2)\subseteq R$ is a complete
intersection of type $(d-1,d-1)$.

The infinitesimal quotient differential of $\Phi(f)$ at $H$ is represented by
\[
\psi(f,x_0):\C^2\to R_d/(J_{\overline f})_d,\qquad
(b_1,b_2)\mapsto [(b_1x_1+b_2x_2)\overline f_0].
\]
Since $J_{\overline f}$ is generated in degree $d-1$, one has
$(J_{\overline f})_d=\overline f_1R_1+\overline f_2R_1$. Thus
\[
\dim_\C R_d/(J_{\overline f})_d=(d+1)-4=d-3.
\]

If $d=3$, then $R_3/(J_{\overline f})_3=0$. Equivalently, every reduced divisor
of degree $3$ on $\PP^1$ is projectively equivalent to every other one. Hence
$\Phi(f)$ is constant. Since $\dim\M(3,1)=0$, this is maximal rank.

Assume now that $d=4$. Then $\dim_\C R_4/(J_{\overline f})_4=1$. Therefore
maximal rank is equivalent to showing that $\psi(f,x_0)$ is nonzero for a
general line $H$. Suppose, to the contrary, that $\psi(f,x_0)=0$ for such a
general line. Then
\[
x_1\overline f_0,\ x_2\overline f_0\in (J_{\overline f})_4.
\]
Let $A=R/J_{\overline f}$. Since $A$ is a complete intersection of type $(3,3)$,
it is an Artinian Gorenstein algebra with socle degree $4$. The two inclusions
above say that the class of $\overline f_0$ in $A_3$ is killed by $A_1$. Thus it
is a socle element in degree $3$. Since the socle of $A$ is concentrated in
degree $4$, this class is zero. Hence $\overline f_0\in (J_{\overline f})_3$,
and so
\[
\overline f_0=\alpha\overline f_1+\beta\overline f_2
\]
for some $\alpha,\beta\in\C$. Lifting to $S$, we get
\[
f_0-\alpha f_1-\beta f_2=x_0k
\]
for some $k\in S_2$.

If $k=0$, then $f_0-\alpha f_1-\beta f_2=0$, giving a constant syzygy among
$f_0,f_1,f_2$. After a linear change of coordinates, this means that $f$ is
independent of one variable, i.e. $C$ is a cone, contrary to the hypothesis.
Thus $k\neq0$. Since $d=4$, the ideal $J_f$ is generated in degree $3$, and hence
$(J_f)_2=0$. Therefore the class of $k$ is nonzero in $M(f)_2$, while the
displayed identity gives $x_0k\in (J_f)_3$. Hence multiplication by the general
linear form $x_0$ is not injective:
\[
x_0:M(f)_2\to M(f)_3.
\]
This contradicts Lemma~\ref{lem:DP-plane-lefschetz}, since for $d=4$ one has $2<3(d-2)/2=3$. Therefore
$\psi(f,x_0)\neq0$ for a general line. Since its target is one-dimensional, its
rank is $1$. Hence $\Phi(f)$ has maximal rank. As
$\dim\PP^{2*}=2$ and $\dim\M(4,1)=1$, the map is not generically finite onto its
image.

Finally assume that $d\geq5$. If $\dim\Aut_{\PP^2}(C)>0$, then the induced
action on the dual plane has positive-dimensional general orbits. For
$g\in\Aut_{\PP^2}(C)$ and a general line $H$, one has
\[
C\cap g(H)=g(C\cap H).
\]
Thus the corresponding line sections are projectively equivalent. Hence
$\Phi(f)$ is constant along a positive-dimensional family of general lines, and
so it is not generically finite onto its image.

Conversely, assume that $\dim\Aut_{\PP^2}(C)=0$. Since $C$ is not a cone,
Proposition~\ref{prop:aut-syz} gives
$e=\operatorname{indeg}(\operatorname{Syz}(J_f))>1$. Also, since $d\geq5$, one
has
\[
d-1<\frac{3(d-2)}2.
\]
Therefore, by Lemma~\ref{lem:DP-plane-lefschetz}, multiplication by a general
linear form is injective:
\[
x_0:M(f)_{d-1}\to M(f)_d,
\]
because
\[
d-1<\frac{3(d-2)}2
\]
for $d\geq5$.
Thus both obstructions in Proposition~\ref{prop:local-criterion} are excluded.
Hence $\psi(f,x_0)$ is injective for a general line $H=V(x_0)$. By Lemma~\ref{lem:psi-luna-slice}, this injectivity implies that 
$\dim\operatorname{Im}\Phi(f)=2$. Since $\dim\PP^{2*}=2$, the map $\Phi(f)$
is generically finite onto its image.

The final assertion follows from Corollary~\ref{cor:no-aut-tau}, because for
$n=2$ the bound becomes $\tau(C)<(d-2)(d-1)$.
\end{proof}
\begin{rk}
\label{rk:plane-curve-bourbaki}
For plane curves, the obstruction in Theorem~\ref{thm:plane-curves} has a
homological interpretation in terms of the Bourbaki degree. Recall that in this
case $e=\operatorname{indeg}(\operatorname{Syz}(J_f))$ is the usual minimal
degree of a Jacobian syzygy. By Proposition~\ref{prop:aut-syz}, for a non-cone
reduced plane curve one has $\dim\Aut_{\PP^2}(C)>0$ if and only if $e=1$.
Thus, for $d\geq5$, Theorem~\ref{thm:plane-curves} says that failure of generic
finiteness of the hyperplane-section map is equivalent to the existence of a linear
Jacobian syzygy.

This connects the geometric obstruction with the results of \cite{JNS}. Indeed,
by \cite[Corollary~2.11]{JNS}, a reduced plane curve with $e=1$ is either free
or nearly free. In the terminology of \cite{JNS}, this means that the Bourbaki
degree of $C$ is at most one: Bourbaki degree zero corresponds exactly to free
curves, while Bourbaki degree one corresponds exactly to nearly free curves.
Hence, for plane curves of degree $d\geq5$, every curve whose line sections fail
to vary generically finitely is a free or nearly free curve with a linear
Jacobian syzygy.

Conversely, not every free or nearly free curve is exceptional for $\Phi(f)$.
The obstruction is the stronger condition $e=1$, or equivalently the existence
of a positive-dimensional projective automorphism group. Thus the failure of
maximal variation is detected by the first graded piece of
$\operatorname{Syz}(J_f)$, not merely by freeness or near-freeness.
\end{rk}

\section{Surfaces in \texorpdfstring{$\PP^3$}{P3}}
\label{sec:surfaces}

In this section we consider the case $n=3$. Thus
$S=\C[x_0,x_1,x_2,x_3]$ and $X=V(f)\subseteq\PP^3$ is a surface of degree
$d\geq3$ with at most isolated singularities. The hyperplane-section map is
\[
\Phi(f):\PP^{3*}\dashrightarrow \M(d,2),
\]
and sends a general plane $H$ to the moduli class of the smooth plane curve
$X\cap H\subseteq H\simeq\PP^2$.

A smooth plane curve of degree $d$ is represented by a point of a nonempty open
subset of $\PP(\C[x_1,x_2,x_3]_d)$. Since
$\dim\PP(\C[x_1,x_2,x_3]_d)=\binom{d+2}{2}-1$ and $\dim PGL_3(\C)=8$, one has
\[
\dim \M(d,2)=\binom{d+2}{2}-9.
\]
In particular, $\dim\M(3,2)=1$, while $\dim\M(d,2)\geq3$ for $d\geq4$. Thus, for
cubic surfaces, generic finiteness from $\PP^{3*}$ is impossible and maximal
rank means rank $1$.

The cubic case is exceptional and will be treated directly. We do not use the
immersion criterion in degree $3$; instead, we prove directly that the
infinitesimal quotient differential is nonzero.

\begin{prop}
\label{prop:cubic-surfaces}
Let $X=V(f)\subseteq\PP^3$ be a cubic surface with isolated singularities. If
$X$ is not a cone, then $\Phi(f)$ has maximal rank. In particular, it is not
generically finite onto its image. If $X$ is a cone, then $\Phi(f)$ is constant.
\end{prop}

\begin{proof}
Let $H=V(\ell)$ be a general plane. We choose $\ell$ so that $H\in U_f$ and so
that the Lefschetz injectivity used below holds. After a projective linear
change of coordinates, we may assume that $\ell=x_0$, so $H=V(x_0)$. We keep the
notation $f$ for the transformed equation. Set $R=\C[x_1,x_2,x_3]$ and write
$\overline h=h(0,x_1,x_2,x_3)$ for $h\in S$.

Since $H\in U_f$, the plane cubic $V(\overline f)\subseteq\PP^2$ is smooth.
Thus $J_{\overline f}=(\overline f_1,\overline f_2,\overline f_3)$ is a complete
intersection of type $(2,2,2)$. Hence $A=R/J_{\overline f}$ is Artinian
Gorenstein with socle degree $3$, and $\dim_\C A_3=1$. The infinitesimal
quotient differential of $\Phi(f)$ at $H$ is represented by
\[
\psi(f,x_0):\C^3\to R_3/(J_{\overline f})_3,\qquad
(b_1,b_2,b_3)\mapsto [(b_1x_1+b_2x_2+b_3x_3)\overline f_0].
\]
Since the target is one-dimensional, it is enough to show that $\psi(f,x_0)$ is
not the zero map.

Assume, by contradiction, that $\psi(f,x_0)=0$. Then
\[
x_1\overline f_0,\ x_2\overline f_0,\ x_3\overline f_0\in (J_{\overline f})_3.
\]
Equivalently, the class of $\overline f_0$ in $A_2$ is killed by $A_1$. Hence it
is a socle element in degree $2$. Since the socle of $A$ is concentrated in
degree $3$, this class is zero. Therefore $\overline f_0\in (J_{\overline f})_2$,
and so
\[
\overline f_0=\alpha\overline f_1+\beta\overline f_2+\gamma\overline f_3
\]
for some $\alpha,\beta,\gamma\in\C$. Lifting to $S$, we get
\[
f_0-\alpha f_1-\beta f_2-\gamma f_3=x_0k
\]
for some $k\in S_1$.

If $k=0$, then $f_0-\alpha f_1-\beta f_2-\gamma f_3=0$, giving a nonzero
constant syzygy among the partial derivatives. After a linear change of
coordinates, this means that $f$ is independent of one variable, i.e. $X$ is a
cone, contrary to the hypothesis. Hence $k\neq0$. Since $J_f$ is generated in
degree $2$, one has $(J_f)_1=0$, so the class of $k$ is nonzero in $M(f)_1$. On
the other hand, the displayed identity gives $x_0k\in (J_f)_2$. Thus
multiplication by the general linear form $x_0$ is not injective:
\[
x_0:M(f)_1\to M(f)_2.
\]
This contradicts \cite[Proposition~3.1]{I}, which applies to cubic surfaces
with isolated singularities. Therefore $\psi(f,x_0)\neq0$ for a general plane
$H$.

It follows that the infinitesimal quotient differential has rank $1$ at a
general point. Since $\dim\M(3,2)=1$, the map $\Phi(f)$ has maximal rank. Since
$\dim\PP^{3*}=3$, it cannot be generically finite onto its image.

Finally suppose that $X$ is a cone and has isolated singularities. Then $X$ is
the cone over a smooth plane cubic. For every plane $H$ not passing through the
vertex, projection from the vertex identifies $H$ projectively with the plane of
the base cubic and sends $X\cap H$ to the base cubic. Hence all general
hyperplane sections are projectively equivalent, and $\Phi(f)$ is constant.
\end{proof}

We now recover the smooth surface case. The case $d=3$ is the cubic case above.
For $d\geq4$, the general criterion applies once the required Lefschetz
multiplication map is known. This is supplied by the weak Lefschetz theorem for
height four complete intersections. The following statement improves the general
smooth high-degree corollary in the case $n=3$, and it also occurs in
\cite[Corollary~1, 2)]{B2}.

\begin{thm}
\label{thm:smooth-surfaces}
Let $X=V(f)\subseteq\PP^3$ be a smooth surface of degree $d\geq3$. If $d=3$,
then $\Phi(f)$ has maximal rank but is not generically finite onto its image. If
$d\geq4$, then
\[
\Phi(f):\PP^{3*}\dashrightarrow \M(d,2)
\]
is generically finite onto its image.
\end{thm}

\begin{proof}
If $d=3$, then $X$ is not a cone, and the result follows from
Proposition~\ref{prop:cubic-surfaces}.

Assume now that $d\geq4$. Since $X$ is smooth, $\tau(X)=0$, so
Corollary~\ref{cor:no-aut-tau} gives $\dim\Aut_{\PP^3}(X)=0$. Moreover,
$J_f=(f_0,f_1,f_2,f_3)$ is an Artinian complete intersection of height four,
generated by forms of degree $d-1$. By \cite[Corollary~7.3]{B+}, the Jacobian
algebra $M(f)$ satisfies WLP in degree $d$. Since $n=3$ and $d\geq4$, we are in
the numerical range of Theorem~\ref{thm:main-criterion}. Hence
Theorem~\ref{thm:main-criterion} applies and gives that $\Phi(f)$ is generically
finite onto its image.
\end{proof}

\begin{rk}
\label{rk:smooth-surface-comparison}
Theorem~\ref{thm:smooth-surfaces} recovers the smooth surface case appearing in
\cite[Corollary~1, 2)]{B2}. The only point not covered by the general smooth
corollary obtained from \cite[Theorem~4.5]{BM} is the quartic surface case
$d=4$; here it is supplied by the height-four complete-intersection Lefschetz
result \cite[Corollary~7.3]{B+}.
\end{rk}
\section{Singular hyperplane sections}
\label{sec:singular-hyperplane-sections}

In this section we explain how a suitable singular hyperplane section can be used
to verify the Lefschetz condition appearing in Theorem~\ref{thm:main-criterion}.
Throughout, $X=V(f)\subseteq\PP^n$ is a hypersurface of degree $d\geq3$ with at
most isolated singularities, where $n\geq3$.

Let $H'=V(\ell)\subseteq\PP^n$ be a hyperplane such that
$H'\cap\Sing(X)=\emptyset$, and assume that the section
$W=X\cap H'\subseteq H'\simeq\PP^{n-1}$ is singular. Choose homogeneous
coordinates $y_1,\ldots,y_n$ on $H'$ and let
$g\in R'=\C[y_1,\ldots,y_n]$ be a defining equation of $W$. We write
$J_g=(g_1,\ldots,g_n)$ and $M(g)=R'/J_g$. Let $I_g=J_g^{\rm sat}$ be the
saturation of $J_g$ with respect to the irrelevant ideal of $R'$, and set
\[
s(g)=\operatorname{indeg}(I_g),\qquad
r(g)=\operatorname{indeg}(\operatorname{Syz}(J_g)).
\]
These numbers do not depend on the chosen coordinates on $H'$. Since $W$ is
singular, one has $s(g)>0$. Moreover, $r(g)=0$ if and only if $W$ is a cone.

We shall use the following Lefschetz injectivity theorem of Dimca--Ilardi.

\begin{thm}[{\cite[Theorem~1.11]{DI}}]
\label{thm:DI-injectivity}
Let $X=V(f)\subseteq\PP^n$ be a hypersurface of degree $d$, and let
$H'=V(\ell)$ be a hyperplane such that $H'\cap\Sing(X)=\emptyset$. Assume that
$W=X\cap H'$ is singular, and let $g$ be a defining equation of $W$ in
coordinates on $H'$. Then multiplication by $\ell$ induces an injective map
\[
\ell:M(f)_k\to M(f)_{k+1}
\]
for every
\[
k\leq \min\{d-3+r(g),\,d-3+s(g)\}.
\]
\end{thm}

The point is that the hyperplane $H'$ in Theorem~\ref{thm:DI-injectivity} is not
a general hyperplane for the map $\Phi(f)$, since $X\cap H'$ is singular.
Nevertheless, it can still be used to prove the Lefschetz condition for a general
hyperplane: injectivity of a fixed multiplication map is an open condition on
$\ell\in S_1$. Thus, if one special linear form $\ell$ gives an injective map
$M(f)_{d-1}\to M(f)_d$, then multiplication by a general linear form is also
injective in this degree.

\begin{thm}
\label{thm:singular-section-criterion}
Let $X=V(f)\subseteq\PP^n$ be a hypersurface of degree $d$, with
$n,d\geq3$ and $(n,d)\neq(3,3)$. Assume that $X$ has at most isolated
singularities and that there exists a hyperplane $H'=V(\ell)$ such that
$H'\cap\Sing(X)=\emptyset$ and $W=X\cap H'$ is singular. Let
$g\in R'=\C[y_1,\ldots,y_n]$ be an equation of $W$. Assume that:
\begin{enumerate}
\item $\dim\Aut_{\PP^n}(X)=0$;
\item $r(g)=\operatorname{indeg}(\operatorname{Syz}(J_g))>1$;
\item $s(g)=\operatorname{indeg}(I_g)>1$.
\end{enumerate}
Then
\[
\Phi(f):\PP^{n*}\dashrightarrow \M(d,n-1)
\]
is generically finite onto its image.
\end{thm}

\begin{proof}
By assumptions (2) and (3), we have $r(g)\geq2$ and $s(g)\geq2$. Hence
\[
d-1\leq \min\{d-3+r(g),\,d-3+s(g)\}.
\]
Applying Theorem~\ref{thm:DI-injectivity} with $k=d-1$, we get that
multiplication by the special linear form $\ell$ is injective:
\[
\ell:M(f)_{d-1}\to M(f)_d.
\]
Since injectivity is an open condition on the linear form, multiplication by a
general linear form is also injective in this degree. Assumption (1) excludes the
automorphism obstruction. Therefore Theorem~\ref{thm:main-criterion} applies,
and $\Phi(f)$ is generically finite onto its image.
\end{proof}

\begin{rk}
\label{rk:singular-section-numerical}
The condition $\dim\Aut_{\PP^n}(X)=0$ in Theorem~\ref{thm:singular-section-criterion}
may be replaced by the numerical condition
$\tau(X)<(d-2)(d-1)^{n-1}$, by Corollary~\ref{cor:no-aut-tau}. Similarly, the
condition $r(g)>1$ may be replaced by
\[
\tau(W)<(d-2)(d-1)^{n-2},
\]
again by Corollary~\ref{cor:no-aut-tau}, applied to the hypersurface
$W\subseteq\PP^{n-1}$. Thus the first two obstructions can often be excluded by
Tjurina-number bounds. The remaining condition $s(g)>1$ says that the saturated
Jacobian ideal of the singular hypersurface $W$ contains no linear form; in
geometric terms, the singular scheme of $W$ is not cut out by a hyperplane in
degree one.
\end{rk}

\begin{cor}
\label{cor:nodal-section-criterion}
Let $X=V(f)\subseteq\PP^n$ be a hypersurface of degree $d$ with at most isolated
singularities. Assume that either $d=3$ and $n\geq5$, or $d\geq4$ and
$n\geq3$. Suppose that
\[
\tau(X)<(d-2)(d-1)^{n-1}
\]
and that there exists a hyperplane $H'$ with $H'\cap\Sing(X)=\emptyset$ such
that $W=X\cap H'$ has exactly $n$ nodes in general linear position in
$H'\simeq\PP^{n-1}$. Then
\[
\Phi(f):\PP^{n*}\dashrightarrow \M(d,n-1)
\]
is generically finite onto its image.
\end{cor}

\begin{proof}
By Corollary~\ref{cor:no-aut-tau}, the Tjurina bound on $X$ gives
$\dim\Aut_{\PP^n}(X)=0$. Since $W$ has exactly $n$ nodes, one has $\tau(W)=n$.
In the stated numerical ranges,
\[
n<(d-2)(d-1)^{n-2}.
\]
Indeed, for $d=3$ and $n\geq5$ this is $n<2^{n-2}$, while for $d\geq4$ and
$n\geq3$ it follows from $n<2\cdot3^{n-2}$. Hence
Corollary~\ref{cor:no-aut-tau}, applied to $W\subseteq\PP^{n-1}$, gives
$r(g)>1$. Finally, since the $n$ nodes are in general linear position in
$\PP^{n-1}$, no hyperplane contains their singular scheme, and so $s(g)>1$.
The result follows from Theorem~\ref{thm:singular-section-criterion}.
\end{proof}
\begin{cor}
\label{cor:singular-surfaces-three-nodes}
Let $X=V(f)\subseteq\PP^3$ be a surface of degree $d\geq4$ with at most isolated
singularities. Assume that $\tau(X)<(d-2)(d-1)^2$ and that there exists a plane
$H'$ with $H'\cap\Sing(X)=\emptyset$ such that $W=X\cap H'$ has exactly three
noncollinear nodes. Then
\[
\Phi(f):\PP^{3*}\dashrightarrow \M(d,2)
\]
is generically finite onto its image.
\end{cor}
\begin{proof}
This is Corollary~\ref{cor:nodal-section-criterion} in the case $n=3$, since
three noncollinear nodes in a plane are in general linear position.
\end{proof}

\begin{cor}
\label{cor:general-smooth-hypersurfaces}
Let $X\subseteq\PP^n$ be a general smooth hypersurface of degree $d$. If either
$d=3$ and $n\geq5$, or $d\geq4$ and $n\geq3$, then
\[
\Phi(f):\PP^{n*}\dashrightarrow \M(d,n-1)
\]
is generically finite onto its image.
\end{cor}

\begin{proof}
Since $X$ is smooth, $\tau(X)=0$, so the automorphism obstruction is excluded by
Corollary~\ref{cor:no-aut-tau}. By \cite[Theorem 1.1]{DI0}, a general smooth hypersurface in
the stated ranges admits a hyperplane section with exactly $n$ nodes in general
linear position. The result follows from Corollary~\ref{cor:nodal-section-criterion}.
\end{proof}

\begin{rk}
\label{rk:general-smooth-comparison}
Together with Beauville's cubic threefold result \cite{B},
Corollary~\ref{cor:general-smooth-hypersurfaces} covers the general smooth
cubic case in all dimensions $n\geq4$: the case $n=4$ is Beauville's theorem,
while the cases $n\geq5$ follow from the nodal hyperplane-section criterion.
For $d\geq4$ and $n\geq3$, the corollary gives generic finiteness for a general
smooth hypersurface without requiring the numerical range $d\geq n+2$ appearing
in Corollary~\ref{cor:smooth-high-degree}.
\end{rk}

\begin{ex}[The Schoen quintic]
\label{ex:schoen-quintic}
Consider the quintic threefold $X=V(f)\subseteq\PP^4$ defined by
\[
f=x_0^5+x_1^5+x_2^5+x_3^5+x_4^5-5x_0x_1x_2x_3x_4.
\]
This hypersurface was introduced by Schoen~\cite{Schoen} and plays an important
role in mirror symmetry; see, for instance, \cite{C}. It has $125$ nodes, located
at the points
\[
(1:\alpha_1:\alpha_2:\alpha_3:(\alpha_1\alpha_2\alpha_3)^{-1}),
\qquad \alpha_i^5=1.
\]
Hence $\tau(X)=125$, and
\[
125< (d-2)(d-1)^{n-1}=3\cdot4^3=192.
\]
Therefore Corollary~\ref{cor:no-aut-tau} gives
$\dim\Aut_{\PP^4}(X)=0$.

Let $H'=V(x_0+x_1)$. None of the nodes of $X$ lies on $H'$. Using
$x_1,x_2,x_3,x_4$ as coordinates on $H'$, the section $W=X\cap H'$ is defined by
\[
g=x_2^5+x_3^5+x_4^5+5x_1^2x_2x_3x_4.
\]
A direct computation with {\sc Singular}~\cite{Sing} shows that $W$ has a unique
singular point, namely $q=(1:0:0:0)$,
and that
\[
\tau(W)=\tau(W,q)=13.
\]
Since
\[
13<(d-2)(d-1)^{n-2}=3\cdot4^2=48,
\]
Corollary~\ref{cor:no-aut-tau}, applied to $W\subseteq\PP^3$, gives
$r(g)>1$.

It remains to check that $s(g)>1$. On the affine chart $x_1=1$, the equation of
$W$ at $q$ is
\[
u^5+v^5+w^5+5uvw.
\]
This germ has multiplicity $3$. Thus its Tjurina ideal is generated by elements
of order at least $2$, and it contains no nonzero linear form. Since $q$ is the
only singular point of $W$, the saturated Jacobian ideal $I_g$ contains no linear
form. Therefore $s(g)>1$.

All hypotheses of Theorem~\ref{thm:singular-section-criterion} are satisfied.
Consequently,
\[
\Phi(f):\PP^{4*}\dashrightarrow \M(5,3)
\]
is generically finite onto its image.
\end{ex}


\end{document}